\documentclass[11pt]{amsart}
\usepackage{geometry}                
\usepackage[parfill]{parskip}    
\usepackage{graphicx}
\usepackage{caption}
 \usepackage{subcaption}
\usepackage{amssymb}
\usepackage{epstopdf}
\usepackage[all,knot,poly]{xy}
\usepackage{pgf,tikz}
\usepackage{mathrsfs}
\usetikzlibrary{arrows}
\newtheorem{theorem}{Theorem}[section]
\newtheorem{definition}{Definition}[subsection]
\newtheorem{lemma}{Lemma}[subsection]

\newtheorem{corollary}{Corollary}[subsection]

\newtheorem{example}{Example}[section]

\title{The Alexander Polynomial of a Rational Link}
\author{Mark E. Kidwell}
\author{Kerry M. Luse}

\begin{document}

\begin{abstract}

We relate some terms on the boundary of the Newton polygon of the Alexander polynomial $\Delta(x,y)$ of a rational link to the number and length of monochromatic twist sites in a particular diagram that we call the standard form. Normalize $\Delta(x,y)$ so that no $x^{-1}$ or $y^{-1}$ terms appear, but $x^{-1}\Delta(x,y)$ and $y^{-1}\Delta(x,y)$ have negative exponents, and so that terms of even total degree are positive and terms with odd total degree are negative. If the rational link has a reduced alternating diagram with no self crossings, then $\Delta(-1, 0) = 1$. If the standard form of the rational link has $m$ monochromatic twist sites, and the $j^{\textrm{th}}$ monochromatic twist site has $\hat{q}_j$ crossings, then $\Delta(-1, 0) = \prod_{j=1}^{m}(\hat{q}_j+1)$. Our proof employs Kauffman's clock moves and a lattice for the terms of $\Delta(x,y)$  in which the $y$-power cannot decrease.

\end{abstract}

\maketitle

\section{Introduction}

It is generally easier to compute the reduced Alexander polynomial $\Delta(x,x)$ of a two-component link than the full two-variable polynomial $\Delta(x,y)$. In this paper, however, we demonstrate some features of the full polynomial $\Delta(x,y)$ of a rational link that seem to have no counterpart for the Alexander polynomial $\Delta(x)$ of a rational knot or for $\Delta(x,x)$ of a rational two-component link.

We use Kauffman's lattice of clock states to accomplish this. Each clock state gives a matching of crossings and incident regions in a link diagram, and corresponds to a non-zero term in the determinant of the Alexander matrix as defined by Alexander in 1928. Each of these terms is an entry in the Alexander polynomial (multiplied by $(1-x)$ or $(1-y)$). Under one of Kauffman's clock moves, the power of $x$ will either increase or decrease by one and the power of $y$ will hold fast or else the power of $y$ will increase or decrease by one and the power of $x$ will hold fast.

Both Alexander and Kauffman require the exclusion of two adjacent regions from a link diagram in their calculations. Edges that border one or both of these regions do not support clock moves, so we call them inactive edges. Any edge that does not border one of these regions supports a clock move, so we call it an active edge. We set up a diagram of a rational link where one component (we chose $y$) is an unknot with no self-crossings and as we traverse this component, active and inactive edges alternate.

As mentioned above, every clock move raises or lowers the degree of one variable, so we can classify our clock moves and their supporting edges as ``uppers'' or ``downers.'' In an alternating link diagram, uppers and downers alternate as we traverse a component. We arrange our diagram so that every active edge along the $y$-component is an upper.

As we go from clock state to clock state via clock moves, the corresponding terms of $\Delta(x,y)$ go from one side of the Newton polygon to the other with no backtracking. (``Forward'' and ``sideways'' are the only options.) We focus in particular on the terms that have no $y$-power. The clocked state at the top of Kauffman's lattice is always one of these. The twist sites in our standard form either have one $x$-strand and one $y$-strand (dichromatic) or two $x$-strands (monochromatic). It is the number and length of these monochromatic twist sites that are detected in the entries along one side of the Newton polygon of $\Delta(x,y)$. If there are no monochromatic twist sites, the term corresponding to the clocked state is the only term on that side.

\section{Background}

\subsection{The Alexander matrix and polynomial}\hfill

Consider a two-component link $L$ with diagram $\hat{L}$.  Number each crossing $c_1, \ldots, c_n$ (where $c(L)=n$), and number each region $r_0, \ldots, r_{n+1}$.  Choose an orientation for each component; label one component $x$ and the other component $y$.  Travel around the link on each component placing dots in the two regions to the left of each undercrossing.   Add labels at each crossing: for each component $t$, the quadrants are marked with $t$, $-t$, $1$ and $-1$ as shown in Figure \ref{fig:alex labels} below.   Figure \ref{fig:alex labels and numbering}(A) shows the Whitehead link with Alexander labels.

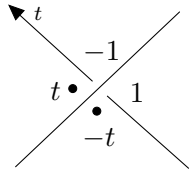
\begin{figure}[h]
\begin{tikzpicture}[line cap=round,line join=round,>=triangle 45,x=1.0cm,y=1.0cm]
\clip(-3.079259259259259,2.1429629629629658) rectangle (2.0762962962962983,5.089629629629634);
\draw (-1.62,2.26)-- (0.56,4.32);
\draw (-0.4,3.2)-- (0.68,2.24);
\draw [->] (-0.6,3.4) -- (-1.72,4.42);
\draw (-1.2925925925925914,3.5162962962963) node[anchor=north west] {$t$};
\draw (-0.8659259259259248,2.9029629629629663) node[anchor=north west] {$-t$};
\draw (-0.2348148148148135,3.498518518518522) node[anchor=north west] {$1$};
\draw (-0.8570370370370358,4.067407407407411) node[anchor=north west] {$-1$};
\begin{scriptsize}
\draw[color=black] (-1.3192592592592582,4.302962962962967) node {$t$};
\draw [fill=black] (-0.86,3.3) circle (1.5pt);
\draw [fill=black] (-0.54,3.) circle (1.5pt);
\end{scriptsize}
\end{tikzpicture}
\caption{Alexander dots and labels for each crossing of an alternating link.}\label{fig:alex labels}
\end{figure}

\begin{figure}[h]
        \centering
         \begin{subfigure}[h]{0.75\textwidth}
  \centering
    \includegraphics[width=.75\textwidth]{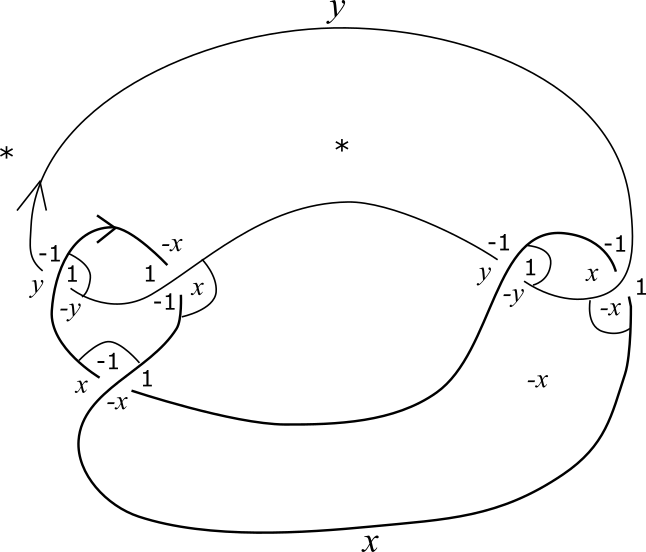}
\caption{An oriented link diagram with Alexander labels.}\label{fig:link with alex labels}
         \end{subfigure}%
           
         \begin{subfigure}[h]{0.75\textwidth}
               \centering
    \includegraphics[width=.75\textwidth]{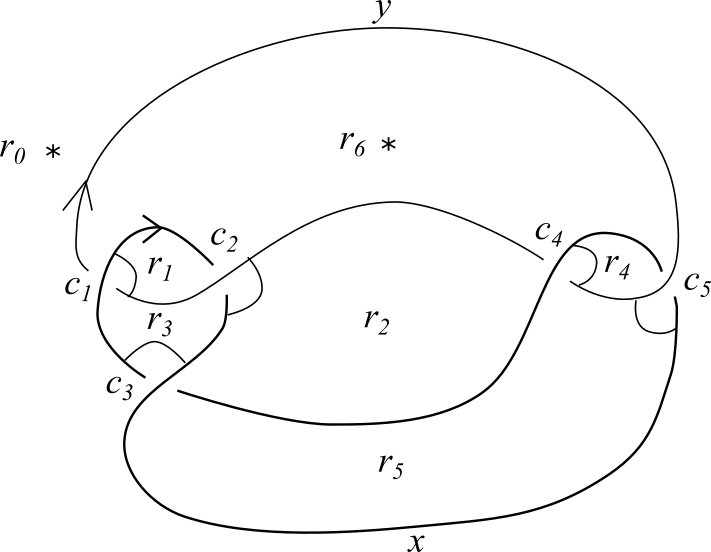}
\caption{Numbering crossings and regions of a link diagram.}\label{fig:link numbering}
         \end{subfigure}
\caption{}\label{fig:alex labels and numbering}
\end{figure}

Next, create an $n\times (n+2)$ matrix with one row per crossing and one column per region.  The entry in the $c_ir_j$ position is the Alexander label described above at crossing $c_i$ and region $r_j$.  If $c_i$ is not incident to $r_j$, the entry is 0.  This matrix is called the \textbf{\emph{Alexander matrix}}.  Choose two adjacent regions (denoted by $*$); these regions will have one edge in common.  We will call this edge the \textbf{\emph{special edge}}.  Now,  strike out the columns corresponding to the chosen regions in the Alexander matrix.  The result is an $n\times n$ matrix, the \emph{\textbf{reduced Alexander matrix}}.  Taking the determinant of the matrix results in $(1-x)\Delta(x,y)$ or $(1-y)\Delta(x,y)$, depending on the label on the special edge, where $\Delta(x,y)$ is the Alexander polynomial which is invariant up to factors of $\pm x^iy^j$.

Numbering the crossings and regions of the Whitehead link as shown in Figure \ref{fig:alex labels and numbering}(B), it follows that the Alexander matrix is:

$$
\begin{bmatrix}
y&1&0&-y&0&0&-1\\
0&1&x&-1&0&0&-x\\
x&0&1&-1&0&-x&0\\
0&0&y&0&1&-y&-1\\
1&0&0&0&x&-x&-1\\
\end{bmatrix}
$$

Strike out the first and last columns of the matrix, corresponding to the two starred regions.  Taking the determinant and normalizing gives:

$$
\begin{vmatrix}
1&0&-y&0&0\\
1&x&-1&0&0\\
0&1&-1&0&-x\\
0&y&0&1&-y\\
0&0&0&x&-x\\
\end{vmatrix}
=
\begin{matrix}
-xy^2&+x^2y^2\\
2xy&-2x^2y\\
-x&x^2\\
\end{matrix}
\approx
\begin{matrix}
y^2&-xy^2\\
-2y&+2xy\\
+1&-x\\
\end{matrix}
=(1-y)
\begin{pmatrix}
-y&xy\\
1&-x\\
\end{pmatrix}
$$

We will follow the convention of Rolfsen \cite{Rolfsen} to display the Alexander polynomial as a matrix of coefficients where the entry in the matrix in the $i^{\textrm{th}}$ column (left to right) and the $j^{\textrm{th}}$ row (bottom to top) is the coefficient of $x^iy^j$ in $\Delta(x,y)$.  For the example above, the matrix is written 
$\left[\begin{smallmatrix}
-1&1\\
1&-1\\
\end{smallmatrix}
\right]$.
This array is also called the Newton polygon of the polynomial.  In this paper, we will show for rational links that the contributions to the ``bottom row'' of the Alexander polynomial can be counted by the number and size of what we call monochromatic twist sites in a rational link.

\subsection{Kauffman's Lattice}\hfill

In \textit{Formal Knot Theory}, \cite{FKT}, Kauffman develops a state sum formula for the calculation of the Alexander polynomial of a link.  Kauffman's Clock Theorem asserts that any two states involved in the state sum are related by a sequence of moves called clock moves.  Before we go into detail about the Clock Theorem, we remind readers of the states of a link diagram.  Flatten an oriented link diagram so that the over/under information is lost and each crossing becomes a vertex.  The result is a 4-valent plane graph, which is called a \textbf{\textit{universe}} (see Figure \ref{fig:proj+univ}).  If we start with a prime knot or link diagram, we call the resulting universe a prime universe.  If the graph is connected, the universe is called a connected universe.  In this paper, we will only consider connected prime universes, and so will use the word universe to mean connected prime universe.

\begin{figure}[h]
  \centering
    \includegraphics[width=.65\textwidth]{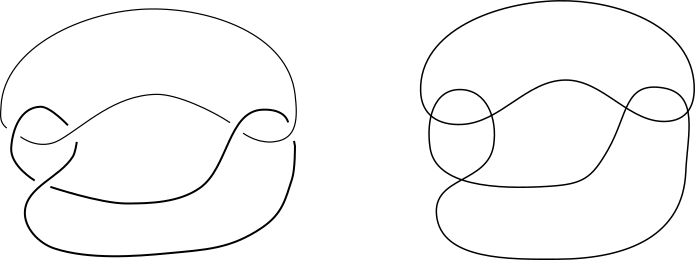}
\caption{A link diagram and a universe.}\label{fig:proj+univ}
\end{figure}

If a universe is labeled with Alexander dots (as described in the previous section), this universe can be used to recover all the crossing information and most of the orientations of the link.  If the universe has a component that splits the dots at each of its crossings, its orientation is not determined.

The number of regions in a link universe is exactly two more than the crossing number, $c(\hat{L})$.  A state of a universe is a selection of two adjacent regions (again denoted by $*$) and an assignment of markers that establishes a one-to-one correspondence between the vertices and the remaining regions (see Figure \ref{fig:example states}).  By the rule for taking determinants, a set of markers also describes one non-zero term in $(1-y)\Delta(x,y)$, assuming that the special edge is part of the $y$ component.

\begin{figure}[h]
        \centering
         \begin{subfigure}[h]{0.5\textwidth}
  \centering
    \includegraphics[width=.5\textwidth]{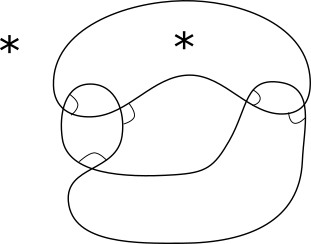}
 \caption{}\label{fig:example statesa}
         \end{subfigure}%
         \begin{subfigure}[h]{0.5\textwidth}
               \centering
    \includegraphics[width=.5\textwidth]{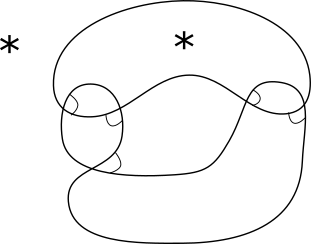}
\caption{}\label{fig:example statesb}
         \end{subfigure}
             \begin{subfigure}[h]{0.5\textwidth}
               \centering
    \includegraphics[width=.5\textwidth]{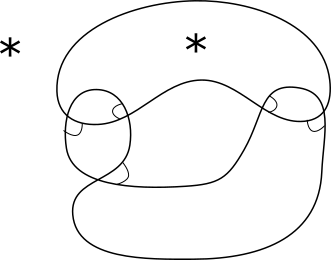}
\caption{}\label{fig:example statesc}
         \end{subfigure}
\caption{States of a universe.}\label{fig:example states}
 \end{figure}

Notice that states $(A)$ and $(B)$ in Figure \ref{fig:example states} differ by exchanging markers in adjacent regions.  In particular, they differ by the move in Figure \ref{fig:clock move}, called a \textbf{\textit{clock move}}.

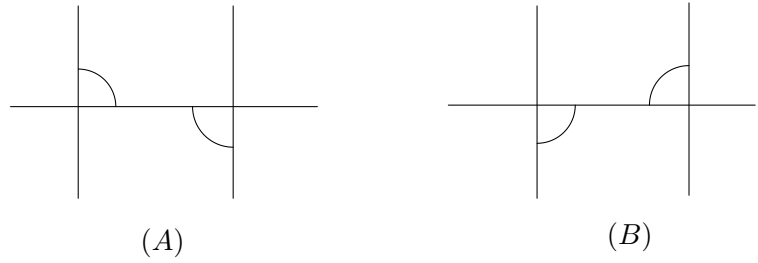
\begin{figure}[h]
\begin{tikzpicture}[line cap=round,line join=round,>=triangle 45,x=1.0cm,y=1.0cm]
\clip(-3.8666666666666676,-0.33) rectangle (7.7333333333333325,4.55);
\draw (-3.52,2.16)-- (0.56,2.16);
\draw (2.3,2.18)-- (6.38,2.18);
\draw (-2.62,3.5)-- (-2.62,0.94);
\draw (-0.56,3.5)-- (-0.56,0.94);
\draw (3.48,3.5)-- (3.48,0.94);
\draw (5.5,3.54)-- (5.5,0.98);
\draw [shift={(-2.62,2.16)}] plot[domain=0.:1.5707963267948966,variable=\t]({1.*0.5*cos(\t r)+0.*0.5*sin(\t r)},{0.*0.5*cos(\t r)+1.*0.5*sin(\t r)});
\draw [shift={(-0.56,2.16)}] plot[domain=3.141592653589793:4.71238898038469,variable=\t]({1.*0.54*cos(\t r)+0.*0.54*sin(\t r)},{0.*0.54*cos(\t r)+1.*0.54*sin(\t r)});
\draw [shift={(3.48,2.18)}] plot[domain=-1.5707963267948966:0.,variable=\t]({1.*0.508*cos(\t r)+0.*0.508*sin(\t r)},{0.*0.508*cos(\t r)+1.*0.508*sin(\t r)});
\draw [shift={(5.5,2.18)}] plot[domain=1.5707963267948966:3.141592653589793,variable=\t]({1.*0.526*cos(\t r)+0.*0.526*sin(\t r)},{0.*0.526*cos(\t r)+1.*0.526*sin(\t r)});
\draw (-1.9266666666666674,0.67) node[anchor=north west] {$(A)$};
\draw (4.2733333333333325,0.75) node[anchor=north west] {$(B)$};
\end{tikzpicture}
\caption{A clock move.}\label{fig:clock move}
\end{figure}

A clock move is associated with an edge.  We will call edges about which clock moves can potentially occur \textbf{\textit{active edges}}.  Thus a clock move is a clockwise $90^{\circ}$ rotation of the two markers at the endpoints of an active edge.  Not all edges in a universe are active edges.  In particular, if an edge is incident to one, or both, of the starred regions then it cannot support a clock move.  We call such an edge an \textbf{\textit{inactive edge}}.  Similarly, we can perform a counterclock move by rotating two markers at the endpoints of an active edge counterclockwise $90^{\circ}$, as in Figure \ref{fig:counterclock move}.

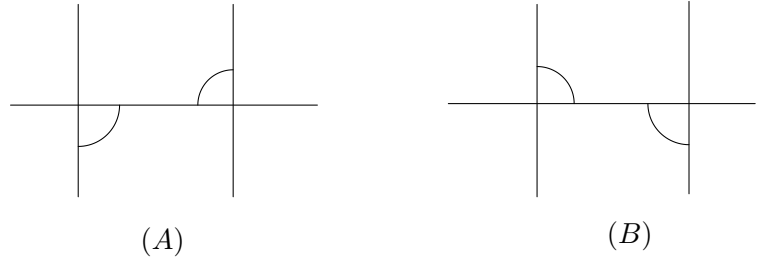
\begin{figure}[h]
\begin{tikzpicture}[line cap=round,line join=round,>=triangle 45,x=1.0cm,y=1.0cm]
\clip(-3.8666666666666676,-0.33) rectangle (7.7333333333333325,4.55);
\draw (-3.52,2.16)-- (0.56,2.16);
\draw (2.3,2.18)-- (6.38,2.18);
\draw (-2.62,3.5)-- (-2.62,0.94);
\draw (-0.56,3.5)-- (-0.56,0.94);
\draw (3.48,3.5)-- (3.48,0.94);
\draw (5.5,3.54)-- (5.5,0.98);
\draw (-1.9266666666666674,0.67) node[anchor=north west] {$(A)$};
\draw (4.2733333333333325,0.75) node[anchor=north west] {$(B)$};
\draw [shift={(-2.62,2.16)}] plot[domain=-1.5707963267948966:0.,variable=\t]({1.*0.55*cos(\t r)+0.*0.55*sin(\t r)},{0.*0.55*cos(\t r)+1.*0.55*sin(\t r)});
\draw [shift={(3.48,2.18)}] plot[domain=0.:1.5707963267948957,variable=\t]({1.*0.4933333333333332*cos(\t r)+0.*0.4933333333333332*sin(\t r)},{0.*0.4933333333333332*cos(\t r)+1.*0.4933333333333332*sin(\t r)});
\draw [shift={(-0.56,2.16)}] plot[domain=1.5707963267948966:3.141592653589793,variable=\t]({1.*0.47*cos(\t r)+0.*0.47*sin(\t r)},{0.*0.47*cos(\t r)+1.*0.47*sin(\t r)});
\draw [shift={(5.5,2.18)}] plot[domain=3.141592653589793:4.71238898038469,variable=\t]({1.*0.5466666666666669*cos(\t r)+0.*0.5466666666666669*sin(\t r)},{0.*0.5466666666666669*cos(\t r)+1.*0.5466666666666669*sin(\t r)});
\end{tikzpicture}
\caption{A counterclock move.}\label{fig:counterclock move}
\end{figure}

The Clock Theorem says that any two states in a connected universe differ by a clock move, or a sequence of clock moves.  Notice in Figure \ref{fig:example states} that states $(B)$ and $(C)$ also differ by a clock move (and $(A)$ and $(C)$ are related by performing both moves).

If a term in the determinant of the reduced Alexander matrix is thought of as a permutation of the numbered rows with the numbered columns, then a clock move represents a transposition of the permutation.

We say that a \textbf{\textit{clocked state}} is a state with only clock moves available and a \textbf{\textit{counterclocked state}} is a state with only counterclock moves available.  In Formal Knot Theory, Kauffman proves:

\begin{theorem}[Theorem 2.5 in \cite{FKT}]
Let $\mathcal U$ be a universe and $\mathcal S$ be the set of states of $\mathcal U$ for a given choice of adjacent fixed stars.  Then:

\begin{enumerate}
\item $\mathcal S$ has a unique clocked state and a unique counterclocked state
\item Any state in $\mathcal S$ can be reached from the clocked (counterclocked) state by a series of clock (counterclock) moves.
\item Any two states in $\mathcal S$ are connected by a series of state transpositions, or clock moves.
\end{enumerate}

\end{theorem}

Thus, given a link universe with a fixed pair of starred (adjacent) regions, one can create a lattice of states with the clocked state at the top and the counterclocked state at the bottom.   There is an arrow between two states in the lattice if they are related by a single clock move (see Figure \ref{fig:lattice example}).  Notice that once an active edge has markers in position to do a clock move, that clock move remains available until it is performed.  Hence, available clock moves can be made in any order and yield the same result (see Lemmas 6 and 7 on page 245 in \cite{Gil-Lith}).

\begin{figure}[h]
  \centering
    \includegraphics[width=.65\textwidth]{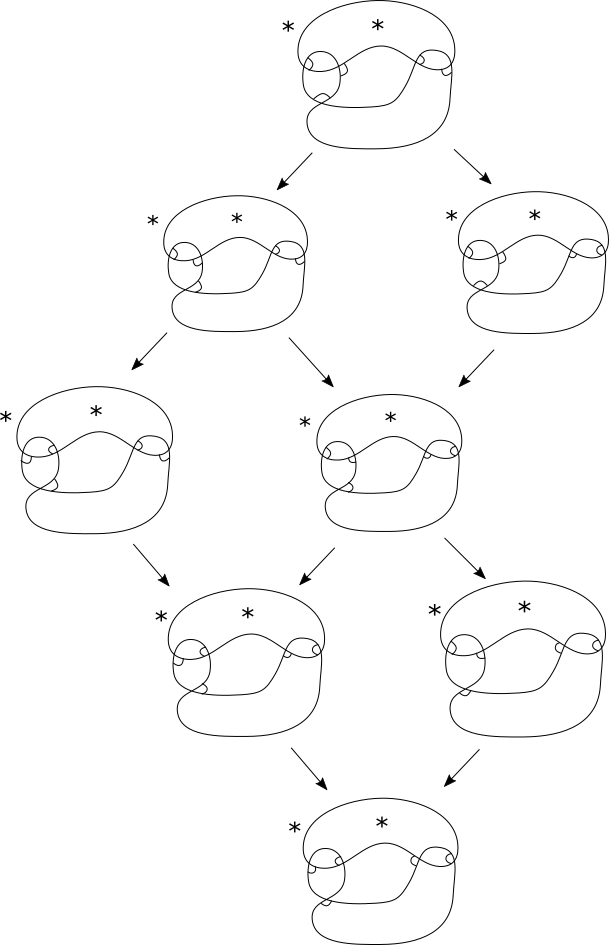}
\caption{A lattice of a universe.}\label{fig:lattice example}
\end{figure}

Since a set of markers describes one non-zero term in $(1-y)\Delta(x,y)$, each state in the lattice corresponds to a term.  We will call the term corresponding to the clocked state the \textbf{\emph{clocked term}}.  Recall Figure \ref{fig:alex labels and numbering}$(B)$ which shows the clocked state of the Whitehead link.  The contribution of the clocked term is the product of the Alexander labels corresponding to the markers in the state.  In this example, the state contributes the term $1\cdot x\cdot -1\cdot 1\cdot -x=x^2$ to the Alexander polynomial.

The Alexander labels on all the corners of a given region in an alternating diagram will have the same sign.  Thus the non-zero terms in any column of the Alexander matrix will have the same sign.  We can remove all the minus signs from the Alexander matrix, with the possible effect of multiplying the determinant of the reduced matrix by $-1$.  Evaluating the determinant of a matrix with all terms $0, 1, x,$ or $y$ does not, however, produce a polynomial with all non-negative coefficients because of the sign convention of the determinant calculation.  Each term represents a permutation of the row and column numbers, and terms that represent an odd permutation are given a minus sign.

These row/column permutations with all non-zero entries correspond exactly to Kauffman's marker states.  We will introduce a way to number our crossings and regions so that the clocked state corresponds to the identity permutation.  Moreover, clock moves change the permutation by a transposition, and thus change the parity.  We have also seen that a clock move raises or lowers the power of one of our variables by one.  Thus if a term has even total degree before a clock move, (of the form $a_{ij}x^iy^j$ with $i+j$ even), it will have odd total degree after the clock move, and vice versa.  We can arrange our Alexander polynomial so that terms of even total degree have positive coefficients and terms of odd total degree have negative coefficients.  One consequence of these observations is that the contribution to $\Delta(x,y)$ from different clock states do not cancel, provided we start with an alternating diagram.  With our conventions, the total number of clock states equals $\Delta(-1,-1)$.

Finally, we need to distinguish between the types of vertices in a link universe.  Consider a link universe with specified starred regions created from a link diagram $\hat{L}$.  There are three types of vertices: special vertices, boundary vertices, and spinners.  The \textbf{\textit{special vertices}} are the two vertices incident to the special edge.  The \textbf{\textit{boundary vertices}} are vertices incident to one of the starred regions.  A \textbf{\textit{spinner}} is a vertex not incident to either of the starred regions, that is, a vertex that is not special nor on the boundary.  Figure \ref{fig:three vertex types} shows the three types of vertices.  This figure shows a piece of a diagram, not a universe, since later we will make use of the over/under information.  The overstrand is still considered to be two edges incident at that crossing.

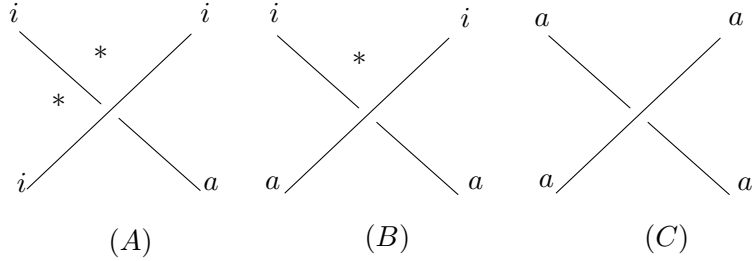
\begin{figure}[h]
\begin{tikzpicture}[line cap=round,line join=round,>=triangle 45,x=1.0cm,y=1.0cm]
\clip(-2.552592592592593,1.2137037037037039) rectangle (9.047407407407412,5.266296296296299);
\draw (-1.62,2.26)-- (0.56,4.32);
\draw (-0.4,3.2)-- (0.68,2.24);
\draw (1.4074074074074088,2.5662962962962976) node[anchor=north west] {$a$};
\draw (0.5474074074074083,4.866296296296299) node[anchor=north west] {$i$};
\draw (0.5874074074074084,2.5462962962962976) node[anchor=north west] {$a$};
\draw (-1.9925925925925925,4.866296296296299) node[anchor=north west] {$i$};
\draw (-1.7192592592592582,4.351851851851856)-- (-0.6392592592592582,3.3918518518518557);
\draw (-0.8725925925925921,4.306296296296298) node[anchor=north west] {$*$};
\draw (1.8074074074074085,2.206296296296296)-- (3.9874074074074066,4.2662962962962965);
\draw (3.0274074074074093,3.1462962962962964)-- (4.107407407407408,2.1862962962962964);
\draw (-1.8925925925925926,2.6062962962962977) node[anchor=north west] {$i$};
\draw (4.00740740740741,4.806296296296298) node[anchor=north west] {$i$};
\draw (4.107407407407409,2.5662962962962976) node[anchor=north west] {$a$};
\draw (1.4674074074074088,4.866296296296299) node[anchor=north west] {$i$};
\draw (1.7081481481481504,4.298148148148152)-- (2.788148148148151,3.338148148148152);
\draw (2.567407407407409,4.226296296296298) node[anchor=north west] {$*$};
\draw (-1.4325925925925924,3.666296296296298) node[anchor=north west] {$*$};
\draw (5.414074074074077,2.2062962962962978)-- (7.594074074074077,4.266296296296298);
\draw (6.634074074074077,3.146296296296298)-- (7.714074074074077,2.186296296296298);
\draw (5.04740740740741,2.5862962962962976) node[anchor=north west] {$a$};
\draw (7.567407407407411,4.746296296296299) node[anchor=north west] {$a$};
\draw (7.687407407407411,2.5462962962962976) node[anchor=north west] {$a$};
\draw (4.98740740740741,4.726296296296298) node[anchor=north west] {$a$};
\draw (5.314814814814819,4.298148148148154)-- (6.3948148148148185,3.3381481481481536);
\draw (-0.6725925925925921,1.8662962962962975) node[anchor=north west] {$(A)$};
\draw (2.727407407407409,1.9262962962962975) node[anchor=north west] {$(B)$};
\draw (6.447407407407411,1.9262962962962975) node[anchor=north west] {$(C)$};
\end{tikzpicture}
\caption{The three types of vertices: special, boundary, and spinner.}\label{fig:three vertex types}
\end{figure}

Clock states and clock moves are easier to analyze at vertices that are incident to one or both of the starred regions.  At one of the special vertices, the clock marker can only be in one of the two unstarred regions.  Furthermore, only one clock move is possible because there is only one active edge incident to such a vertex.  At any other boundary vertex, the clock marker can only be in one of three unstarred regions.  In this instance, only two clock moves are possible.  There are two active edges incident to such a vertex.  In particular, the two active edges are next to each other.

At a vertex that is not incident to a starred region, all four incident edges are active and the clock marker may move freely through the four incident regions.  This is the reason for the name \textbf{\textit{spinner}} for these vertices, although it is the clock marker that does the actual spinning. It is an important feature of rational knots and links that they have diagrams with no spinners.

\subsection{Conway conventions}\hfill

In this section, we introduce the conventions we follow regarding rational tangles (see \cite{Conway}).  We will draw our rational tangles in a herringbone pattern starting at the northwest (NW) quadrant of the tangle and ending at the southeast (SE) quadrant (Figure \ref{fig:standard knot}).  In addition:

\begin{enumerate}
\item The first and last twist sites have at least two crossings.\label{tangle cond 1}
\item The twist sites alternate between horizontal and vertical.  This rule defines whether a twist site with a single crossing is horizontal or vertical. \label{tangle cond 2}
\item All horizontal twist sites are left-turning and all vertical twist sites are right-turning.\label{tangle cond 3}
\item The last (SE) twist site is horizontal.\label{tangle cond 4}
\end{enumerate}

\begin{figure}[h]
        \centering
         \begin{subfigure}[h]{0.75\textwidth}
  \centering
    \includegraphics[width=.75\textwidth]{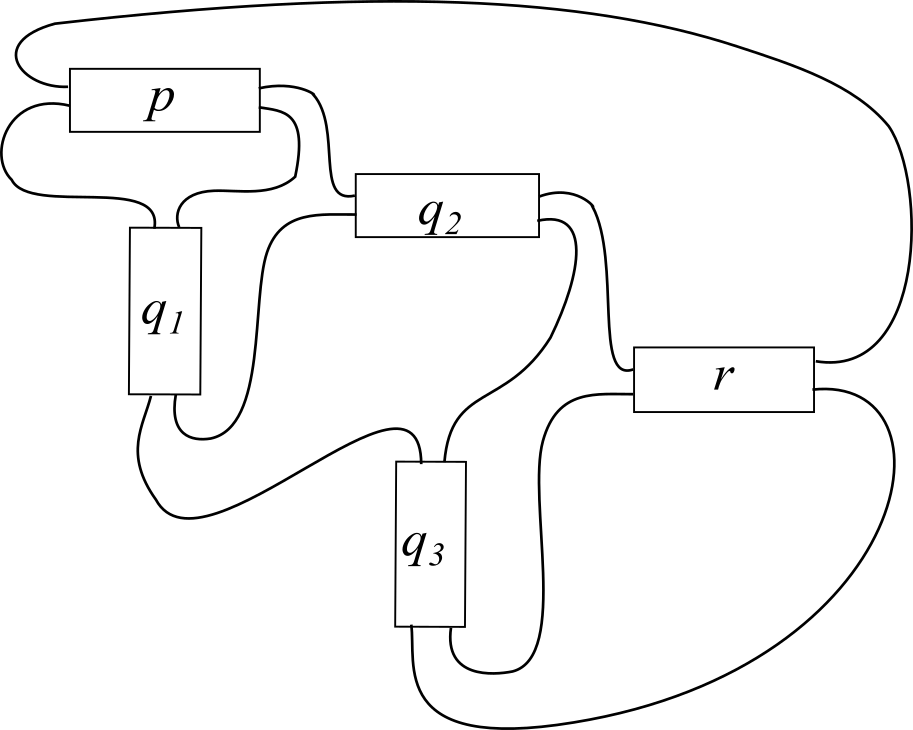}
 \caption{}\label{fig:standard knot horiz}
         \end{subfigure}%
          
         \begin{subfigure}[h]{0.75\textwidth}
               \centering
    \includegraphics[width=.75\textwidth]{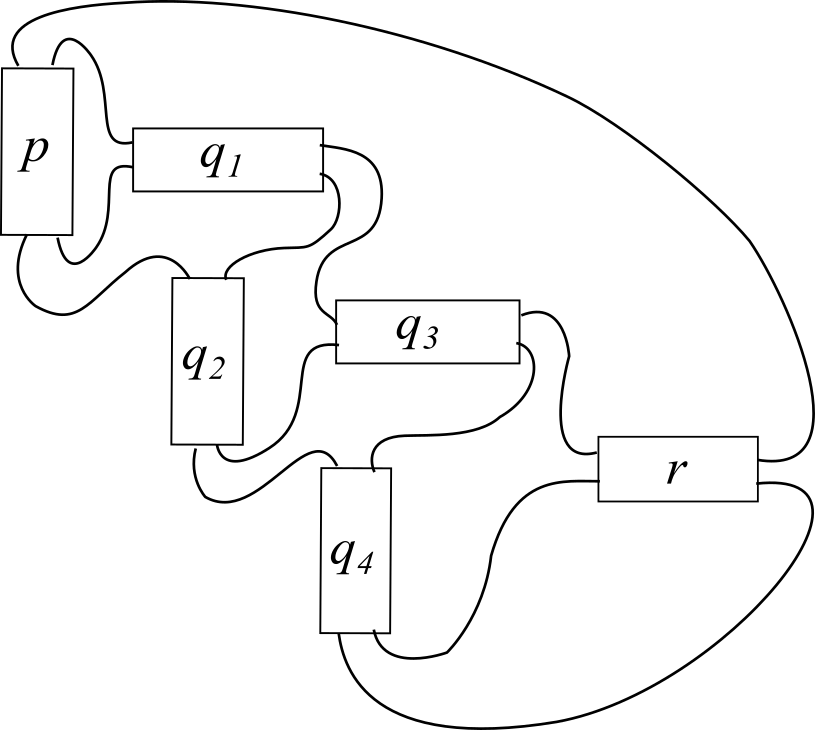}
\caption{}\label{fig:standard knot vert}
         \end{subfigure}
\caption{Standard diagrams of a rational link.}\label{fig:standard knot}
 \end{figure}

These rules have a number of simple consequences. Rule \ref{tangle cond 3} ensures that our tangle diagrams are alternating.  At every crossing, the overcrossing segment has positive slope.  In keeping with Rule \ref{tangle cond 4}, we will always use the ``numerator closure,'' $N(T)$, (NW strand joined to NE, SW strand joined to SE) when making our tangles into rational knots or links.  By ``horizontal'' and ``vertical''  twist sites, we mean integer tangles of the form shown in Figure \ref{fig:horiz and vert sites}.

\begin{figure}[h]
    \begin{subfigure}[c]{0.5\textwidth}
    \centering
\includegraphics[width=.5\textwidth]{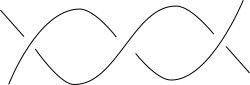}
                 \label{fig:horiz twist site}
         \end{subfigure}
    \begin{subfigure}[c]{0.25\textwidth}
	\includegraphics[width=.25\textwidth]{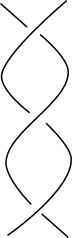}
                 \label{fig:vert twist site}
         \end{subfigure}
\caption{Horizontal and vertical twist sites are integer tangles.}\label{fig:horiz and vert sites} 
\end{figure}

We will write the Conway notation (see \cite{Conway}) for a rational tangle with three or more twist sites as $pq_1q_2\ldots q_kr$, where $p\geq2$, $r\geq2$, and $q_i\geq1$ for $i=1,\ldots,k$ and $p$ is the number of crossings in the first twist site, $q_i$ is the number of crossings in the $i^{\textrm{th}}$ internal twist site, and $r$ is the number of crossings in the final twist site. Note that the total number of crossings is $n=p+\sum_{i=1}^kq_i+r$.  There are $k+2$ total twist sites in such a tangle. For ease of notation, we will refer to the location of the site by its length.

This notation also allows us to number the crossings and regions in a consistent manner.  Beginning with the first twist site, number these crossings $c_1, c_2, \ldots c_p$ from left to right or top to bottom, depending on whether the first twist site is horizontal or vertical.  Continue the numbering ($c_{p+1},\ldots, c_{p+q_1}$) in the next twist site, again moving left to right or top to bottom, and so on until the final crossing, $c_n$.  The two special vertices are labeled $c_1$ and $c_n$.  

For the regions, we label the unbounded region $r_0$ and the region enclosed by the NW/NE strand of the numerator closure as $r_{n+1}$.  Notice that all of the crossings in vertical twist sites are incident to $r_{0}$ and all of the crossings in horizontal twist sites are incident to $r_{n+1}$. 

Assign $r_0$ and $r_{n+1}$ to be the starred regions and place markers for the clocked state in the diagram.  The clocked state can be achieved by placing markers at each crossing in the region located clockwise to the starred region(s) at that crossing. Now label the regions so that the clocked state corresponds to the main diagonal of the reduced (square) Alexander matrix. That is, at crossing $c_1$, look at the adjacent region that has a marker in the clocked state and call that $r_1$. Then do the same for $c_2$ and $r_2$, and so forth.  This numbering scheme was used to label the Whitehead link in Figure \ref{fig:alex labels and numbering}$(B)$.  We will always use this procedure to number the crossings and regions for diagrams of rational links in this form.  

The pictures in Figure \ref{fig:standard knot} will be referred to as a \emph{``herringbone''} diagram of a knot or link.  These diagrams follow the conventions given by Conway in \cite{Conway}.  However, one can imagine pulling on the right hand side of the diagram until all of the vertical twist sites appear to be horizontal.  The diagram will then look like the standard 4-plat diagram of a rational knot or link.

The numerator closure of a rational tangle will produce either a knot or a link depending on the tangle.  In this paper we are focusing on rational links which have two components.  Thus, it is important to be able to tell when a rational tangle will close to a knot or to a link.  Figure \ref{fig:tangle strands} below illustrates three possibilities for the strands of our rational tangle.  For example, in Figure \ref{fig:tangle strands}$(A)$, the strand at the NW post goes through the tangle and exits at the SE post, in Figure \ref{fig:tangle strands}$(B)$, the strand at the NW post goes through the tangle and exits at the SW post.  When taking the numerator closure, the two tangles in Figure \ref{fig:tangle strands}$(A)$ and $(B)$ will close to knots while the tangle in Figure \ref{fig:tangle strands}$(C)$ will close to a 2-component link.

\begin{figure}[h]
\begin{tikzpicture}[line cap=round,line join=round,>=triangle 45,x=0.65cm,y=0.65cm]
\clip(-5.105,-2) rectangle (10.995,4.45);
\draw(-1.8,1.82) ellipse (1.0973779031913835cm and 1.1642402278336579cm);
\draw(3.68,1.8) ellipse (1.0973779031913835cm and 1.1642402278336579cm);
\draw(8.73,1.85) ellipse (1.0973779031913835cm and 1.164240227833658cm);
\draw (-3.81,3.65)-- (0.21,-0.25);
\draw (-0.18,3.8)-- (-3.66,-0.4);
\draw [shift={(8.73,-0.76)}] plot[domain=0.4699766384170082:2.665523323267032,variable=\t]({1.*2.119834899231541*cos(\t r)+0.*2.119834899231541*sin(\t r)},{0.*2.119834899231541*cos(\t r)+1.*2.119834899231541*sin(\t r)});
\draw [shift={(8.7,4.16)}] plot[domain=3.527929868411108:5.984686375593407,variable=\t]({1.*1.9108375127153014*cos(\t r)+0.*1.9108375127153014*sin(\t r)},{0.*1.9108375127153014*cos(\t r)+1.*1.9108375127153014*sin(\t r)});
\draw [shift={(6.18,1.79)}] plot[domain=1.9815132899095687:4.270796618077399,variable=\t]({1.*2.0287188075236053*cos(\t r)+0.*2.0287188075236053*sin(\t r)},{0.*2.0287188075236053*cos(\t r)+1.*2.0287188075236053*sin(\t r)});
\draw [shift={(0.96,1.82)}] plot[domain=-1.009302663527798:1.0557235097761148,variable=\t]({1.*2.1973848092675983*cos(\t r)+0.*2.1973848092675983*sin(\t r)},{0.*2.1973848092675983*cos(\t r)+1.*2.1973848092675983*sin(\t r)});
\draw (-2.3,-1) node[anchor=north west] {$(A)$};
\draw (3,-1) node[anchor=north west] {$(B)$};
\draw (8.3,-1) node[anchor=north west] {$(C)$};
\end{tikzpicture}
\caption{Three possibilities for tangle strands.}\label{fig:tangle strands}
\end{figure}
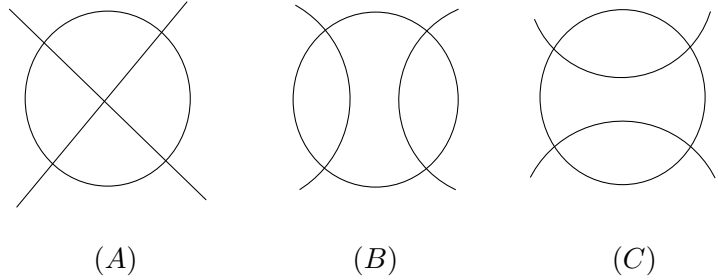

\subsection{Consequences of Conway's conventions}\hfill

When traveling along a knot or link diagram, we can use colors/labels to keep track of the components.  Visually, it is nice to picture links with different color strands.  For this reason, in this paper we will refer to the components as the ``green'' and ``red'' components.  For calculation purposes, the components need labels.  We use the label $y$ for the green strand and $x$ for the red strand.  At any given twist site, we will see crossings either between strands of two different colors or strands with the same color. 

\begin{definition} A \underline{\bf{monochromatic}} twist site is a twist site whose crossings consist of strands of a single color/label.
\end{definition}

\begin{definition} A \underline{\bf{dichromatic}} twist site is a twist site whose crossings consist of strands of two colors/labels.
\end{definition}

\begin{lemma}\label{Lem mono x label} For a rational tangle in herringbone form, any monochromatic twist sites that exist will all have the same color/label.
\end{lemma}

\begin{proof} (With thanks to Hugh Morton.) A rational 2-tangle consists of two strands.  The strands have four endpoints which we will call the NW, NE, SE, and SW posts of the tangle.  We label the strand with a NW endpoint with $y$  and the other strand with an $x$.  We think of these labels as also applying to the posts, so there are two $x$-posts and two $y$-posts in any 2-tangle.

We construct our rational tangles according to Conway's herringbone pattern.  When constructing a horizontal twist site, we twist the NE and SE posts around each other.  If the two strands involved have different labels and we construct a crossing site with an odd number of crossings, the labels on the NE and SE posts are switched.  Similarly, when constructing a vertical twist site, we twist the SE and SW posts around each other, with a possible switch of labels.

The NW post is not involved in either of these operations, so it retains its $y$ label.  Only one other post can have a $y$-label, so the twisting operations that create new twist sites must involve two $x$-labels or an $x$- and a $y$-label.  Thus, monochromatic $x$-labeled twist sites and dichromatic twist sites are possible, but monochromatic $y$-labeled twist sites are not possible.

\end{proof}

We will call the strand going into the tangle at the NW post and leaving the tangle at the NE post the $y$-strand and the strand going into the tangle at the SW post and leaving the tangle at the SE post is the $x$-strand.  The numerator closure will create a rational two-component link.   Furthermore, the $y$-component is an unknot in standard position (that is, no monochromatic $y$-crossings).  Hence, by Lemma \ref{Lem mono x label}, all monochromatic twist sites will have an $x$-label.

\begin{corollary}\label{Cor: first and last} In a rational link, both the initial and final twist sites must be dichromatic.
\end{corollary}

\begin{proof}  For a rational link with labels described above, we know that there is a strand with a $y$-label at the NW and NE post, which corresponds to the initial and final twist sites.  By Lemma \ref{Lem mono x label}, any twist site which has a $y$-label must also have an $x$-label.  Hence these two twist sites must be dichromatic.
\end{proof}

The following two lemmas are consequences of the way the twist sites in a rational tangle are connected to each other.

\begin{lemma}\label{Lem:no consec mono} A rational link $L$, in herringbone form, cannot have two consecutive monochromatic twist sites.
\end{lemma}

\begin{proof}  Two consecutive twist sites in a rational link in herringbone form will have two strands (one from each site) that connect to the next consecutive twist site.  Suppose $\hat{L}$ has $k+2$ twist sites and the last two internal twist sites ($q_{k-1}$ and $q_k$) are monochromatic (with label $x$).  The NE strand of $q_{k-1}$ and the SE strand of $q_k$ will twist around each other to form $r$, see Figure \ref{consec sites}$(A)$.  Since the internal sites were monochromatic, both of the strands creating the final twist site have an $x$ label, making the final twist site monochromatic, contradicting Corollary \ref{Cor: first and last}.  Therefore, $\hat{L}$ cannot have two consecutive monochromatic twist sites immediately preceding the last twist site.

\begin{figure}[h]
        \centering
         \begin{subfigure}[h]{0.65\textwidth}
  \centering
    \includegraphics[width=.65\textwidth]{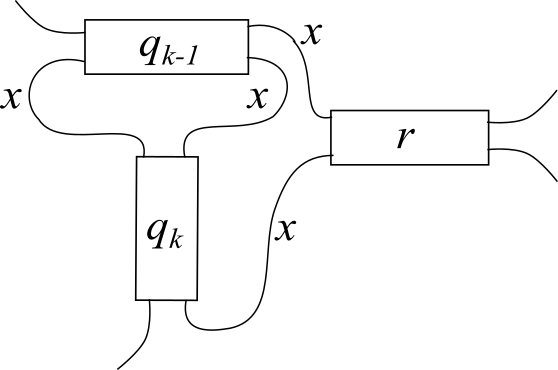}
 \caption{}\label{consec sites a}
         \end{subfigure}%
         \begin{subfigure}[h]{0.5\textwidth}
               \centering
    \includegraphics[width=.5\textwidth]{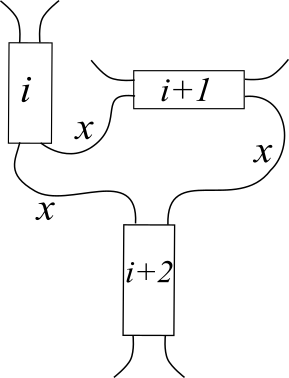}
\caption{}\label{consec sites b}
         \end{subfigure}
    \begin{subfigure}[h]{0.65\textwidth}
               \centering
    \includegraphics[width=.65\textwidth]{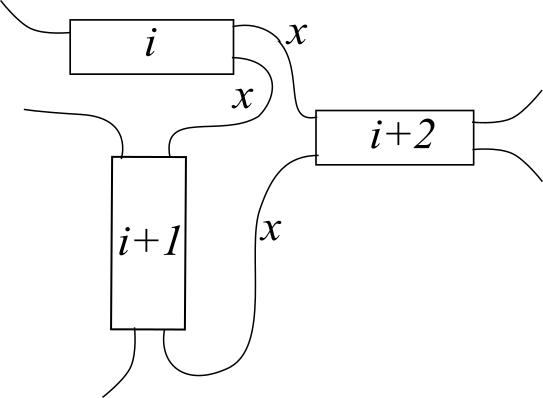}
\caption{}\label{consec sites c}
         \end{subfigure}
 \caption{Twist sites and the connecting strands between them.}\label{consec sites}
 \end{figure}
 
Now suppose that $\hat{L}$ has two consecutive monochromatic internal twist sites (with label $x$), sites $i$ and $i+1$.  By the above argument, these twist sites cannot immediately precede the last twist site.  However, suppose also that they are the right-most such internal twist sites in the herringbone diagram.  These two twist sites will each have one strand which twist around each other to create twist site $i+2$, see Figures \ref{consec sites} $(B)$and \ref{consec sites}$(C)$.  Since sites $i$ and $i+1$ were monochromatic, both of the strands creating twist site $i+2$ have an $x$ label, making site $i+2$ monochromatic.  However, this means that sites $i+1$ and $i+2$ are a pair of consecutive monochromatic twist sites, further to the right than $i$ and $i+1$, a contradiction.

It follows that $\hat{L}$ cannot have two consecutive monochromatic twist sites.

\end{proof}

Hence, for a rational link $pq_1q_2\ldots q_kr$ in herringbone form we know that $p$ and $r$ are dichromatic (Corollary \ref{Cor: first and last}), and possible monochromatic twist sites can occur at the $q_i$ sites, though $q_i$ and $q_{i+1}$ cannot both be monochromatic (Lemma \ref{Lem:no consec mono}).  For a rational link with $m$ monochromatic twist sites, we will use $\hat{q}_1,\ldots,\hat{q}_m$ to specify the monochromatic sites (again, $\hat{q}_i$ will be both its position and its length).  For example, if $m=2$ with $q_2$ and $q_4$ as the monochromatic sites, we will write the link as $pq_1\hat{q}_1q_3\hat{q}_2q_5r$.

A rational link will have many of its active edges within twist sites.  There are also active edges between the twist sites.  In particular, the edges that connect horizontal and vertical twist sites play an important role.

\begin{definition}  An \underline{\bf{HV connector edge}} is an active edge of a rational tangle which connects a horizontal twist site and the next consecutive vertical twist site.
\end{definition}

\begin{definition}  A \underline{\bf{VH connector edge}} is an active edge of a rational tangle which connects a vertical twist site and the next consecutive horizontal twist site.
\end{definition}

\begin{lemma}\label{lemma HV red}  In a rational link in herringbone form, an HV connector edge has an $x$ label if and only if exactly one of the twist sites it connects is monochromatic.
\end{lemma}

\begin{proof}
Suppose exactly one of the twist sites $i$ and $i+1$ is monochromatic and that the HV connector edge between these two twist sites is green (with a $y$-label).  It follows that the $y$ strand twists through both the $i$ and $i+1$ twist sites.   By Lemma \ref{Lem mono x label}, any twist site which has a $y$-label must also have an $x$-label.  Hence both twist sites must be dichromatic, a contradiction.  It follows that if exactly one of the twist sites is monochromatic, the HV connector edge must have an $x$-label.

Now suppose there exists a red HV connector edge between two dichromatic twist sites.  We will consider two cases.  

\underline{Case 1}:  The two twist sites are the sites preceding the final twist site as in Figure \ref{fig:case 1}.

\begin{figure}[h]
  \centering
    \includegraphics[width=.45\textwidth]{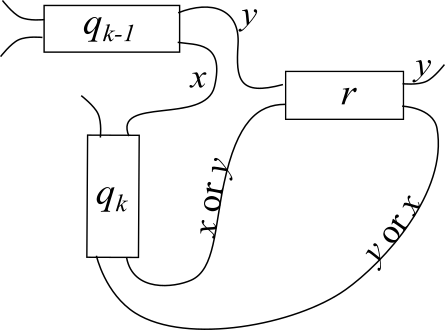}
 \caption{Case 1.}\label{fig:case 1}
 \end{figure}
 
The HV connector edge has an $x$-label, and we know that sites $q_{k-1}$ and $q_k$ are dichromatic.  This means that the NE post of $q_{k-1}$ must have a $y$-label and the SW and SE posts of $q_k$ will have one $x$-label and one $y$-label.  We also know that since this is a link diagram, the NE post of $r$ will have a $y$-label.  It follows that $r$ has three $y$-labels and only one $x$-label, which is a contradiction. Case 1 cannot occur.

\underline{Case 2}: The two twist sites are internal at positions $i$ and $i+1$ as in Figure \ref{fig:case 2}. 

\begin{figure}[h]
   \centering
    \includegraphics[width=.45\textwidth]{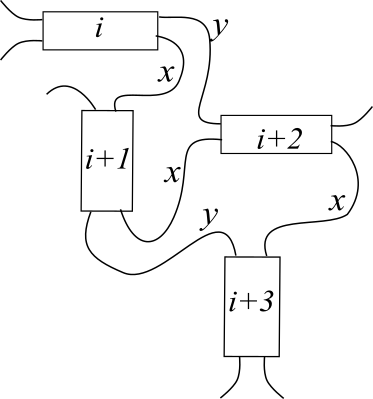}
 \caption{Case 2.}\label{fig:case 2}
 \end{figure}
 
Assume that these are the right-most pair of twist sites in the diagram that  are dichromatic with a red HV connector edge.  We know the HV connector edge has an $x$-label and since site $i$ is dichromatic, we know that the label at its NE post must be $y$.  Furthermore, the label at the SE post of $i+1$ must be $x$, or else these two strands will create a monochromatic $y$ site at $i+2$, which is not possible.  Since $i+1$ is dichromatic, the SW post must have a $y$-label.  This strand goes into the NW post of site $i+3$, giving this post a $y$-label.  This site cannot be monochromatic $y$, and so the NE post must have an $x$-label.  This creates a red ($x$-labeled) HV connector edge between $i+2$ and $i+3$, which are both dichromatic.  This contradicts our assumption that we had the right-most pair of such twist sites.  It follows that a red HV connector edge cannot connect two dichromatic twist sites.

Since we cannot have two consecutive monochromatic sites, it follows that a red HV connector edge connects exactly one monochromatic and one dichromatic twist site. 

\end{proof}

\section{Results}

We will now make more connections between the Alexander polynomial and clocked states of a rational link.  The over- and under-crossing information will be useful and so instead of placing clock markers on a link universe, they will be placed on the diagram.  Edges and vertices of the link universe are replaced by edges and crossings of the link diagram.  An overstrand has two pieces; that is, each half of the overstrand is either an active or inactive edge.

Recall that a rational link in herringbone form ensures that there is a choice of adjacent regions where all of the crossings are incident to one of the starred regions (the special crossings are incident to both).  In order to create the Alexander matrix, the link diagram must have a label and an orientation on each component, and a choice of starred regions.  When the link is in herringbone form with crossings and regions numbered as in Figure \ref{fig:link numbering}, we will always specify $r_0$ and $r_{n+1}$ to be the starred regions.

\begin{definition} A labeled and oriented rational link with a herringbone diagram is in \underline{\bf{standard form}} if:

\begin{itemize}
\item The link has one component labeled $y$ which is an unknotted circle and the special edge is part of the $y$ component.
\item The link has one component labeled $x$ which may or may not have self-crossings.
\item At crossing $c_1$, the undercrossing ($y$ component) orientation is pointing NW and the overcrossing ($x$ component) orientation is pointing NE.
\end{itemize}

\end{definition}

Lemma \ref{a and i alternate} states an important consequence of standard form.  This lemma is trivially true for the Hopf link.

\begin{lemma}\label{a and i alternate} For a rational link in standard form, the edges of the $y$-component alternate between active and inactive.  The $x$-component has two inactive edges in a row at each special vertex, otherwise they alternate.
\end{lemma}

\begin{proof} Recall Figure \ref{fig:three vertex types}, which shows the three possible types of crossings.  Since the rational link is in standard form, there are two special crossings, $n-2$ boundary crossings, and no spinners.  

Notice that at a boundary crossing, the active edges are adjacent to each other.  However, when you travel along one component of the link, you go ``straight'' at each crossing.  This means that along one component, away from the special vertices, the component will alternate between active and inactive edges.  By definition of standard form, the $y$-component contains the special edge connecting the two special vertices.  In Figure \ref{fig:three vertex types}$(A)$, this is the edge labeled $i$ which has stars on both sides.  Therefore, when traveling along the $y$-component, there is a single inactive edge connecting the two special crossings and so the edges of the $y$-component alternate between active and inactive, as indicated in Figure \ref{fig:y strand alt}.

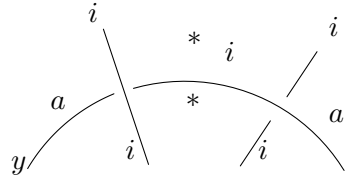
\begin{figure}[h]
\begin{tikzpicture}[line cap=round,line join=round,>=triangle 45,x=1.0cm,y=1.0cm]
\clip(-4.3,3.3) rectangle (4.3,6.3);
\draw [shift={(-0.34,2.34)}] plot[domain=1.9636668342595445:2.6069983301487976,variable=\t]({1.*2.4030813552603663*cos(\t r)+0.*2.4030813552603663*sin(\t r)},{0.*2.4030813552603663*cos(\t r)+1.*2.4030813552603663*sin(\t r)});
\draw [shift={(-0.32,2.24)}] plot[domain=0.5500743726751895:1.8480797259872657,variable=\t]({1.*2.4868453912537465*cos(\t r)+0.*2.4868453912537465*sin(\t r)},{0.*2.4868453912537465*cos(\t r)+1.*2.4868453912537465*sin(\t r)});
\draw (-0.42,5.52) node[anchor=north west] {$*$};
\draw (-2.76,3.86) node[anchor=north west] {$y$};
\draw (-2.24,4.66) node[anchor=north west] {$a$};
\draw (1.46,5.72) node[anchor=north west] {$i$};
\draw (-0.44,4.7) node[anchor=north west] {$*$};
\draw (1.46,4.48) node[anchor=north west] {$a$};
\draw (0.52,4.08) node[anchor=north west] {$i$};
\draw (-1.24,4.08) node[anchor=north west] {$i$};
\draw (-1.74,5.9) node[anchor=north west] {$i$};
\draw (0.08,5.38) node[anchor=north west] {$i$};
\draw (-1.4,5.42)-- (-0.8,3.62);
\draw (1.44,5.12)-- (1.04,4.52);
\draw (0.82,4.2)-- (0.42,3.6);
\end{tikzpicture}
\caption{The $y$-component alternates between active and inactive edges.}\label{fig:y strand alt}
\end{figure}

On the other hand, the $x$-component will travel straight across each of the special crossings.  That is, the $x$-component will have two inactive edges in a row as it crosses each special crossing.  Away from the special crossings, the edges alternate. See Figure \ref{fig:x strand almost alt}.

\begin{figure}[h]
\begin{tikzpicture}[line cap=round,line join=round,>=triangle 45,x=1.0cm,y=1.0cm]
\clip(-4.3,3.3) rectangle (4.3,6.3);
\draw [shift={(0.1,1.58)}] plot[domain=1.0648459275238205:2.1025203940537027,variable=\t]({1.*3.383962174729499*cos(\t r)+0.*3.383962174729499*sin(\t r)},{0.*3.383962174729499*cos(\t r)+1.*3.383962174729499*sin(\t r)});
\draw [shift={(-0.46,2.38)}] plot[domain=2.1611030737302688:2.601173153319209,variable=\t]({1.*2.4074052421642684*cos(\t r)+0.*2.4074052421642684*sin(\t r)},{0.*2.4074052421642684*cos(\t r)+1.*2.4074052421642684*sin(\t r)});
\draw [shift={(0.66,2.72)}] plot[domain=0.4678492649476779:0.9151007005533605,variable=\t]({1.*2.128755504984074*cos(\t r)+0.*2.128755504984074*sin(\t r)},{0.*2.128755504984074*cos(\t r)+1.*2.128755504984074*sin(\t r)});
\draw (-2.18,5.08)-- (-1.04,3.56);
\draw (2.24,5.26)-- (1.42,3.64);
\draw (0.08,5.66)-- (0.1,5.08);
\draw (0.1,4.76)-- (0.1,3.72);
\draw (-3,3.9) node[anchor=north west] {$x$};
\draw (-1.16,5.58) node[anchor=north west] {$*$};
\draw (0.78,5.62) node[anchor=north west] {$*$};
\draw (-0.76,5.44) node[anchor=north west] {$i$};
\draw (1.14,5.36) node[anchor=north west] {$i$};
\draw (0.06,6.12) node[anchor=north west] {$i$};
\draw (-2.14,5.5) node[anchor=north west] {$i$};
\draw (2.16,5.64) node[anchor=north west] {$i$};
\draw (-1.1,4.04) node[anchor=north west] {$a$};
\draw (0.08,4.12) node[anchor=north west] {$a$};
\draw (1.44,4.06) node[anchor=north west] {$a$};
\draw (-2.54,4.48) node[anchor=north west] {$a$};
\draw (2.22,4.54) node[anchor=north west] {$a$};
\end{tikzpicture}
\caption{The $x$-component alternates between active and inactive edges, except when crossing through a special crossing, where it travels over two inactive edges in a row.}\label{fig:x strand almost alt}
\end{figure}
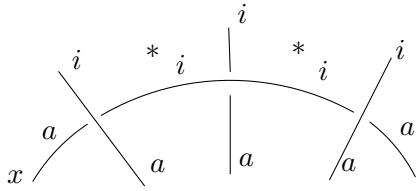
\end{proof}

Returning to the consideration of the Alexander dots, it is important to understand how each state in Kauffman's lattice contributes to the Alexander polynomial.  Recall Figure \ref{fig:alex labels} which showed the clocked state contributes $1\cdot x\cdot -1\cdot 1\cdot -x=x^2$ to $(1-y)\Delta(x,y)$.  We can also determine the contribution from any given state if we know how many $x$ and $y$ Alexander dots are enclosed by the state markers.  That is, we can create a term $\pm x^iy^j$ and use the convention that the term is positive if $i+j$ is even and negative if $i+j$ is odd.  See Figure \ref{fig:state contr ex} below.  Since there are two dots from the $x$-component (and no other dots) within state markers, the contribution is $x^2$.

\begin{figure}[h]
   \centering
    \includegraphics[width=.35\textwidth]{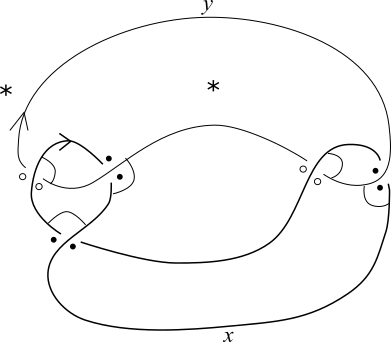}
\caption{An oriented link diagram with Alexander dots and clock markers.}\label{fig:state contr ex}
\end{figure}

When the markers of a state change with a clock move, the contribution changes as well.  In particular, each active edge in the diagram can be considered as an \textbf{\textit{upper}} or a \textbf{\textit{downer}}.  When a clock move occurs, the markers at the endpoints of the active edge will have one endpoint (specifically the undercrossing endpoint) which either gains (upper) or loses (downer) an Alexander dot.  The other endpoint (the overcrossing endpoint)  will remain neutral (either remain empty or remain dotted).  Since the orientation of the undercrossing determines the placement of the Alexander dots, it is the active edge that determines the dots at the crossing in question.  That is, the dots and the active edge will be of the same color/label.

It follows that in an alternating diagram, when the orientation of the active edge travels from the overcrossing to the undercrossing, the active edge is an upper, and when the orientation of the active edge travels from the undercrossing to the overcrossing, the active edge is a downer (see Figure \ref{fig:up and down}).  This figure shows two edges, one of which is an upper and the other a downer.  However, only one of the edges will be an active edge of the diagram.  That is, we can declare all edges in a diagram as uppers or downers, but it is only the active edges that will contribute to the power of $x$ and $y$ in the polynomial.  We will make use of the fact that uppers and downers alternate in Lemma \ref{green are uppers}.

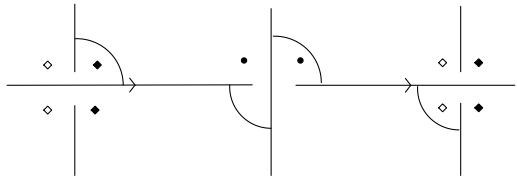
\begin{figure}[h]
\begin{tikzpicture}[line cap=round,line join=round,>=triangle 45,x=1.5cm,y=1.5cm]
\clip(-2.0,2.0) rectangle (3.1,3.8);
\draw (-1.88,2.9)-- (0.29251665657498827,2.903878070367597);
\draw (-0.7460959586397359,2.9020240855917896) -- (-0.793632321187945,2.8406802612330937);
\draw (-0.7460959586397359,2.9020240855917896) -- (-0.7938510222370666,2.963197809134502);
\draw (0.68,2.9)-- (2.62,2.9);
\draw (1.6976457889827072,2.9) -- (1.65,2.838741128450805);
\draw (1.6976457889827072,2.9) -- (1.65,2.961258871549194);
\draw (-1.28,3.62)-- (-1.28,3.02);
\draw (-1.28,2.72)-- (-1.28,2.1);
\draw (0.46,3.58)-- (0.46,2.1);
\draw (2.14,3.6)-- (2.14,3.02);
\draw (2.14,2.74)-- (2.14,2.1);
\draw [shift={(-1.2632509350022234,2.901100933457045)}] plot[domain=0.0017850570842782717:1.6024936980735798,variable=\t]({1.*0.41238392414888775*cos(\t r)+0.*0.41238392414888775*sin(\t r)},{0.*0.41238392414888775*cos(\t r)+1.*0.41238392414888775*sin(\t r)});
\draw [shift={(0.46,2.8857590574716014)}] plot[domain=3.0932194708252934:4.71238898038469,variable=\t]({1.*0.36735093413922343*cos(\t r)+0.*0.36735093413922343*sin(\t r)},{0.*0.36735093413922343*cos(\t r)+1.*0.36735093413922343*sin(\t r)});
\draw [shift={(0.49252842496487775,2.923098016720637)}] plot[domain=0.0017850570842782704:1.6024936980735802,variable=\t]({1.*0.412383924148888*cos(\t r)+0.*0.412383924148888*sin(\t r)},{0.*0.412383924148888*cos(\t r)+1.*0.412383924148888*sin(\t r)});
\draw [shift={(2.1251842954871254,2.871518114943203)}] plot[domain=3.0932194708252934:4.71238898038469,variable=\t]({1.*0.3673509341392238*cos(\t r)+0.*0.3673509341392238*sin(\t r)},{0.*0.3673509341392238*cos(\t r)+1.*0.3673509341392238*sin(\t r)});
\begin{scriptsize}
\draw [fill=black] (0.22,3.12) circle (1.0pt);
\draw [fill=black] (0.72,3.12) circle (1.0pt);
\draw [color=black] (-1.52,3.08) ++(-1.5pt,0 pt) -- ++(1.5pt,1.5pt)--++(1.5pt,-1.5pt)--++(-1.5pt,-1.5pt)--++(-1.5pt,1.5pt);
\draw [color=black] (-1.52,2.68) ++(-1.5pt,0 pt) -- ++(1.5pt,1.5pt)--++(1.5pt,-1.5pt)--++(-1.5pt,-1.5pt)--++(-1.5pt,1.5pt);
\draw [fill=black] (-1.08,3.08) ++(-1.5pt,0 pt) -- ++(1.5pt,1.5pt)--++(1.5pt,-1.5pt)--++(-1.5pt,-1.5pt)--++(-1.5pt,1.5pt);
\draw [fill=black] (-1.1,2.68) ++(-1.5pt,0 pt) -- ++(1.5pt,1.5pt)--++(1.5pt,-1.5pt)--++(-1.5pt,-1.5pt)--++(-1.5pt,1.5pt);
\draw [color=black] (1.98,3.1) ++(-1.5pt,0 pt) -- ++(1.5pt,1.5pt)--++(1.5pt,-1.5pt)--++(-1.5pt,-1.5pt)--++(-1.5pt,1.5pt);
\draw [color=black] (1.98,2.7) ++(-1.5pt,0 pt) -- ++(1.5pt,1.5pt)--++(1.5pt,-1.5pt)--++(-1.5pt,-1.5pt)--++(-1.5pt,1.5pt);
\draw [fill=black] (2.3,3.1) ++(-1.5pt,0 pt) -- ++(1.5pt,1.5pt)--++(1.5pt,-1.5pt)--++(-1.5pt,-1.5pt)--++(-1.5pt,1.5pt);
\draw [fill=black] (2.3,2.7) ++(-1.5pt,0 pt) -- ++(1.5pt,1.5pt)--++(1.5pt,-1.5pt)--++(-1.5pt,-1.5pt)--++(-1.5pt,1.5pt);
\end{scriptsize}
\end{tikzpicture}
\caption{In an alternating diagram, the direction of the active edge determines when a clock move is an upper or a downer.}\label{fig:up and down}
\end{figure}

By identifying the active edges as either uppers or downers, we can determine the effect of the clock moves on the Alexander polynomial. For example, a clock move performed on an upper $x$ edge will increase the power on the $x$ variable in that state's contribution.  Our conventions for standard form guarantee that a rational link has a diagram for which the clocked term will only contain $y^0$.  Furthermore, we will show that we can perform clock moves on active $x$-edges only, creating a sub-lattice of states that does not contain any $y$ factors.  These states contribute to the bottom row of $(1-y)\Delta(x,y)$. This row is identical to the bottom row of $\Delta(x,y)$.


\begin{lemma}\label{green are uppers}  For a rational link in standard form, all active green edges are uppers.
\end{lemma}

\begin{proof}
For a rational link $L$, suppose diagram $\hat{L}$ is in standard form with crossings and regions numbered in the usual way.  Travel along the $y$-component of $\hat{L}$ starting at crossing $c_1$ and following the orientation of the component.  The first edge encountered is inactive because it is the special edge between the two starred regions.  Because $\hat{L}$ is in standard form, this edge is oriented from an undercrossing to an overcrossing ($c_n$).  Hence, this edge is also a downer (see Figure \ref{fig:up and down}).  By Lemma \ref{a and i alternate}  we know that the $y$-component alternates between inactive and active edges.  Furthermore, in an alternating link uppers and downers must alternate as well.  Since the special edge is inactive and a downer, it follows that all inactive edges on the $y$-component are downers and all active edges on the $y$-component are uppers.
\end{proof}

Notice that since the $x$-component has two inactive edges in a row at each special vertex, the $x$-compoent will have both upper and downer active edges.  

\begin{corollary}\label{Cor:no y}  The clocked term in the Alexander polynomial will not have any $y$-factors.
\end{corollary}

\begin{proof} Suppose we have a link diagram with Alexander dots and markers indicating the clocked state.  A marker in the clocked state will be in a corner that is also dotted when the active edge that has an undercrossing at the given crossing is a downer; the marker will start out in a dotted corner and as it rotates through clock moves it will end in a corner that is not dotted.  Similarly, when the marker is in a corner that is not dotted, the active edge that undercrosses at the given crossing will be an upper.  Since all of the active $y$-edges are uppers, the clocked term in the Alexander polynomial will not have any $y$-factors.
\end{proof}

\begin{lemma}\label{HV available}Let $L$ be a rational link with diagram $\hat{L}$ in standard form.  The clocked state of $\hat{L}$ will have clock moves available on all HV connector edges.
\end{lemma}

\begin{proof}
Suppose $L$ is a rational link with diagram $\hat{L}$ in standard form, numbering crossings and regions as usual.  All crossings in the diagram have overcrossings with positive slope.  For each crossing we can determine which quadrant (see Figure \ref{fig:vertex quads}) the marker will be in for the clocked state.

\begin{figure}[h]
\begin{tikzpicture}[line cap=round,line join=round,>=triangle 45,x=1.0cm,y=1.0cm]
\clip(-4.58,2.0) rectangle (7.02,4.24);
\draw (0.,2.3)-- (1.42,4.);
\draw (0.,4.)-- (0.58,3.32);
\draw (0.86,3.)-- (1.42,2.3);
\draw (0.5,4.04) node[anchor=north west] {$N$};
\draw (0.92,3.5) node[anchor=north west] {$E$};
\draw (0.44,3.06) node[anchor=north west] {$S$};
\draw (-0.04,3.44) node[anchor=north west] {$W$};
\end{tikzpicture}
\caption{Quadrant labels for a typical crossing in a rational link in standard form}\label{fig:vertex quads}
\end{figure}
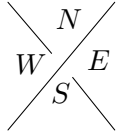

Recall that the markers in the clocked state are placed clockwise to the starred region(s).  So, at crossing $c_1$, the marker will be in the E quadrant, regardless of whether $p$ is horizontal or vertical.  At crossing $c_n$, the marker will be in the S quadrant.  For horizontal boundary vertices, the marker is in the E quadrant and for vertical boundary vertices the marker is in the N quadrant.  Since $\hat{L}$ is in herringbone form, the horizontal twist sites are all adjacent to $r_{n+1}$ and the vertical twist sites are all adjacent to $r_0$.  Hence $c_2, \ldots, c_{n-1}$ are all boundary vertices.  It follows that any HV connector edge will have a clock marker at each endpoint, see Figure \ref{fig:HV markers}.  These markers are in position to perform a clock move on the HV connector edge.  It follows that all HV connector edges have available clock moves in the clocked state.

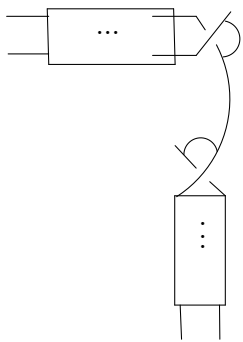
\begin{figure}[h]
\begin{tikzpicture}[line cap=round,line join=round,>=triangle 45,x=1.0cm,y=1.0cm]
\clip(-5.25,-1.0) rectangle (5.4,4.8);
\draw (-1.76,3.76)-- (-0.08,3.76);
\draw (-0.08,3.76)-- (-0.06,3.02);
\draw (-0.06,3.02)-- (-1.74,3.02);
\draw (-1.74,3.02)-- (-1.76,3.76);
\draw (-1.74,3.66)-- (-2.3,3.66);
\draw (-1.74,3.16)-- (-2.28,3.16);
\draw (0.68,3.72)-- (0.22,3.14);
\draw (-0.36,3.14)-- (0.22,3.14);
\draw (-0.36,3.66)-- (0.22,3.66);
\draw (0.22,3.66)-- (0.34,3.48);
\draw [shift={(-0.88,2.56)}] plot[domain=-0.9971296586267364:0.48503472185253566,variable=\t]({1.*1.5477725931156683*cos(\t r)+0.*1.5477725931156683*sin(\t r)},{0.*1.5477725931156683*cos(\t r)+1.*1.5477725931156683*sin(\t r)});
\draw (-1.25,3.6) node[anchor=north west] {...};
\draw (0.22,1.64)-- (-0.06,1.94);
\draw (0.4,1.46)-- (0.6072727272727273,1.2727272727272734);
\draw [shift={(0.56,3.36)}] plot[domain=-1.4876550949064562:1.373400766945016,variable=\t]({1.*0.24083189157584564*cos(\t r)+0.*0.24083189157584564*sin(\t r)},{0.*0.24083189157584564*cos(\t r)+1.*0.24083189157584564*sin(\t r)});
\draw [shift={(0.28,1.82)}] plot[domain=0.17985349979247844:3.141592653589793,variable=\t]({1.*0.22360679774997896*cos(\t r)+0.*0.22360679774997896*sin(\t r)},{0.*0.22360679774997896*cos(\t r)+1.*0.22360679774997896*sin(\t r)});
\draw (-0.06545454545454535,1.2727272727272734)-- (0.6072727272727273,1.2727272727272734);
\draw (0.6072727272727273,1.2727272727272734)-- (0.6072727272727273,-0.18181818181818074);
\draw (0.6072727272727273,-0.18181818181818074)-- (-0.06545454545454535,-0.18181818181818074);
\draw (-0.06545454545454535,-0.18181818181818074)-- (-0.06545454545454535,1.2727272727272734);
\draw (0.12,1.0594788765248147) node[anchor=north west] {.};
\draw (0.12,0.8929888286297833) node[anchor=north west] {.};
\draw (0.12,0.7473100367216309) node[anchor=north west] {.};
\draw (0.007272727272727431,-0.18181818181818074)-- (0.025454545454545556,-0.6363636363636351);
\draw (0.53,-0.18181818181818074)-- (0.5345454545454547,-0.6363636363636351);
\end{tikzpicture}
\caption{Clock markers on an HV connector edge.}\label{fig:HV markers}
\end{figure}

\end{proof}

By Lemma \ref{HV available}, all HV connector edges are available for performing a clock move, regardless of the label on the edge.  Notice that a similar argument would show that a VH connector edge will never have markers in the clocked state for which a clock move is available.  Since our goal is to maximize the number of $x$-edge moves performed, we are only concerned with red HV connector edges.  By Lemma \ref{lemma HV red}, we know that red HV connector edges are connected to exactly one monochromatic twist site.  It follows that when a clock move is performed on a red HV connector edge, it will trigger a \textbf{\emph{cascade}} of clock moves through the monochromatic site.  See, for example, Figure \ref{fig:cascade}.  

\begin{figure}[h]
   \centering
    \includegraphics[width=.25\textwidth]{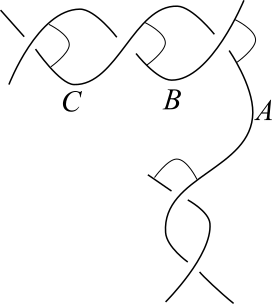}
\caption{An HV connector edge will set off a cascade of available clock moves.}\label{fig:cascade}
\end{figure}

Completing the clock move at edge A will make a clock move available on edge B and completing the move on edge B will make a move on edge C available.  All of these active edges will be red ($x$) edges, and so the corresponding states will not contain any $y$'s; this will be true for any red HV connector edge and the resulting cascade of red active edges in the monochromatic twist site.

In the clocked state, there are also available clock moves in the first and last twist sites of a rational link.  In particular, the final active edge in $r$ will always have a clock move available.  However, if the link is in standard form, this edge will be green ($y$).  Similarly, if $p$ is vertical, the first active edge in that twist site will have a clock move available; again, this edge will be green.  When $p$ is horizontal, the available clock move is on the HV connector edge as explained above.  This edge will either be green, and so we won't use it, or it will be red and will set off a cascade as explained above. 

We can now state our main result:

\begin{theorem}[Main Result]  Let $L$ be a rational link with reduced, alternating diagram $\hat{L}$ in standard form with $m$ monochromatic twist sites.  If $m=0$, then $\Delta(-1,0)=1$.  If $m>0$ and site $\hat{q}_j$ has $\hat{q}_j$ crossings, then $\Delta(-1, 0) = \prod_{j=1}^{m}(\hat{q}_j+1)$. 
\end{theorem}

\begin{proof}
The clocked state for $\hat{L}$ will not contain any $y$'s, by Corollary \ref{Cor:no y}.  When we consider possible available clock moves, we know that HV connector edges are available by Lemma \ref{HV available}.  If $m=0$, then there are no red HV connector edges by Lemma \ref{lemma HV red}.  This means that the clocked state is the only state which does not contain any $y$'s.  It follows that $\Delta(-1,0)=1$.

On the other hand, if $m>0$, there are red HV connector edges.  We can perform all red HV connector moves without increasing the $y$ degree.  Furthermore, each red HV connector edge will start a cascade of red active edges.  Since clock moves can be performed in any order, it remains to show how many different ways we can perform the red HV connector moves along with the moves in each cascade.

For each monochromatic twist site $\hat{q}_j$, $j=1,\ldots,m$, let $x_j$ represent the number of active edges in that twist site that have been used.  For example, consider the monochromatic site in Figure \ref{fig:cascade} with active edges B and C within the twist site.  For this particular site, we can use just edge B or both edges B and C (as they become available).  So $x_j$ can be 0, 1, 2, or 3.  Here, if $x_j=0$, this means that move $A$ has been completed. If $x_j=1$ it represents using the edge B not C, since C cannot be performed without performing the move at edge B first.

Now create the set $\left\{ (x_1,\ldots,x_m), 0\leq x_j\leq \hat{q}_j\right\}$.  Each element of this set represents some combination of active edges that have been used in each of the monochromatic twist sites, including the corresponding HV connector edge.  The size of this set is $(\hat{q}_1+1)(\hat{q}_2+1)\cdots(\hat{q}_m+1)$.  It follows that there will be this many terms in $(1-y)\Delta(x,y)$ that contain only the $x$ variable.  Therefore, $\Delta(-1, 0) = \prod_{j=1}^{m}(\hat{q}_j+1)$.

\end{proof}

We illustrate our main result with the following example.  

\begin{example}

\end{example}
Consider the rational link (with Conway notation 221122) in Figure \ref{fig:bwex2edges} below.  

\begin{figure}[h]
   \centering
    \includegraphics[width=.65\textwidth]{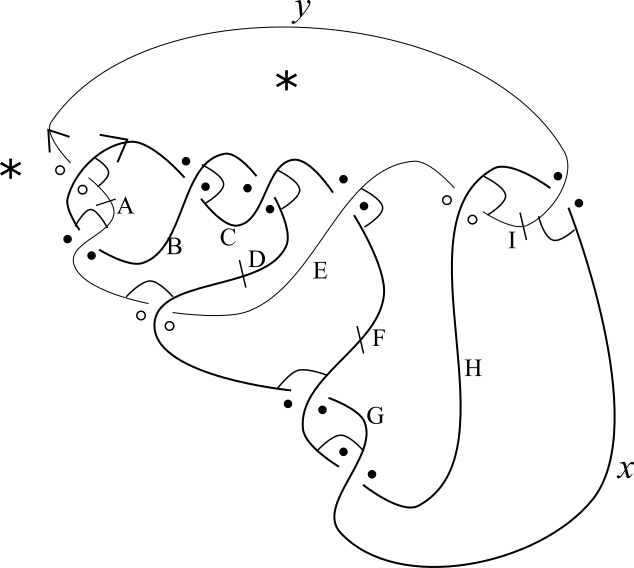}
\caption{The link 221122 with Alexander dots, clock markers, and edge labels.}\label{fig:bwex2edges}
\end{figure}

The active edges in this diagram are labeled $A$ through $I$; edges $A$, $D$, $F$, and $I$ have clock moves available in the clocked state.  Edges $A$ and $I$ are part of the $y$ component, and so we will not perform these moves. Edges $D$ and $F$ are both HV connector edges, each connected to exactly one monochromatic twist site.  Performing a clock move on edge $D$ will begin a cascade in that a clock move on edge $C$ will become available.  Similarly, performing the move on edge $F$ will result in a move on edge $G$ becoming available.  Clock moves at $D$ and $F$ can be done in any order, though we cannot perform $C$ without $D$, nor $G$ without $F$.  

From the placement of the Alexander dots, the clocked state contributes $x^3$ to the bottom row of $(1-y)\Delta(x,y)$.  Moves on edges $D$ and $G$ increase the $x$-degree, moves on edges $C$ and $F$ decrease the $x$-degree. The lattice in Figure \ref{ex2lattice} shows the contribution of each state in the bottom row of $(1-y)\Delta(x,y)$.

\begin{figure}[h]
\centering
\begin{tikzpicture}
\node (C) {$x^3$};
\node (L11)[node distance=1.5cm, left of=C, below of=C]{$x^4$};
\node (L12)[node distance=1.5cm, right of=C, below of=C]{$x^2$};
\node (L21)[node distance=1.5cm, left of=L11, below of=L11]{$x^3$};
\node (L22)[node distance=1.5cm, right of=L11, below of=L11]{$x^3$};
\node (L23)[node distance=1.5cm, right of=L12, below of=L12]{$x^3$};
\node (L31)[node distance=1.5cm, left of=L22, below of=L22]{$x^4$};
\node (L32)[node distance=1.5cm, right of=L22, below of=L22]{$x^2$};
\node (CC)[node distance=1.5cm, right of=L31, below of=L31]{$x^3$};
\draw[->] (C) to node[above] {$D$} (L11);
\draw[->] (C) to node[above] {$F$} (L12);
\draw[->] (L11) to node[above] {$C$} (L21);
\draw[->] (L11) to node[above] {$F$} (L22);
\draw[->] (L12) to node[above] {$D$} (L22);
\draw[->] (L12) to node[above] {$G$} (L23);
\draw[->] (L21) to node[above] {$F$} (L31);
\draw[->] (L22) to node[above] {$C$} (L31);
\draw[->] (L22) to node[above] {$G$} (L32);
\draw[->] (L23) to node[above] {$D$} (L32);
\draw[->] (L31) to node[above] {$G$} (CC);
\draw[->] (L32) to node[above] {$C$} (CC);
\end{tikzpicture}
\caption{}
\label{ex2lattice}
\end{figure}

There are nine states in this lattice.  This link has two monochromatic twist sites, each of length two and $(2+1)(2+1)=9$, as expected.

\end{document}